\documentclass[reqno,10pt]{article} 
\usepackage{amsmath,amsthm,amssymb,amscd,verbatim}
\usepackage{graphicx}
\usepackage{graphicx}
\usepackage{epsfig}

\usepackage{color}

\usepackage{pgf,tikz}
\usepackage{mathrsfs}
\usetikzlibrary{arrows}
\usepackage[final]{pdfpages}
\usepackage{pstricks-add}

\newcommand{\ignore}[1]{}

\input{epsf}

\newtheorem{prelem}{{\bf Theorem}}

\newtheorem{theorem}{Theorem}
\newtheorem{corollary}[theorem]{Corollary}
\newtheorem{definition}[theorem]{Definition}

\newtheorem{proposition}[theorem]{Proposition}

\newtheorem{remarka}[theorem]{Remark}

\newtheorem{examplea}[theorem]{Example}

\newtheorem{exercisea}[theorem]{Exercise}

\def\span{{\mathrm{span}}}

\def\wsat{{\mathrm{wsat}}}
\def\codim{{\mathrm{co\mbox{-}dim}}}

\providecommand{\keywords}[1]{\textbf{\textit{Keywords:}} #1}

\def\b{}
\date{}

\begin{document}

\title{\b{Lower bounds for graph bootstrap percolation via properties of polynomials}}

\author{Lianna Hambardzumyan \thanks{School of Computer Science, McGill University. \texttt{lianna.hambardzumyan@mail.mcgill.ca}} \and Hamed Hatami \thanks{School of Computer Science, McGill University. \texttt{hatami@cs.mcgill.ca}. Supported by an NSERC grant.} \and Yingjie Qian \thanks{Department of Mathematics and Statistics, McGill University. \texttt{yingjie.qian@mail.mcgill.ca}}}

\maketitle 

\begin{abstract}
We introduce a simple method for proving lower bounds for the size  of the smallest percolating set in a certain graph bootstrap process.  We apply this method to \b{obtain recursive formulas for} the sizes of the smallest percolating sets in multidimensional tori and multidimensional grids (in particular hypercubes). The former  answers a question of   Morrison and Noel~\cite{MorrisonNoel}, and the latter provides an alternative and simpler proof for one of  their main results. 
\end{abstract}


\section{Introduction}
Graph bootstrap processes  arise  naturally in statistical mechanics, probability theory, combinatorics, and social sciences, and thus have been extensively studied  in the  past four decades or so. In these processes, one starts with an initial set of infected vertices (\b{or} sites) or edges (\b{or} bonds) in a graph, and at every step, the infection spreads to a new vertex or edge according to some local rule. The goal is often to understand the properties of the \emph{percolating} sets, i.e. the initial sets of infected vertices or edges for which the infection eventually spreads to all the vertices or edges.

The most commonly studied notion of bootstrap percolation is the \emph{$r$-neighbour bootstrap percolation}, introduced in~\cite{Chalupa} in the context of disordered magnetic systems in statistical mechanics. In this process, one starts with an initial set of infected vertices, and at every step, the infection spreads to the vertices that have at least $r$ infected neighbours. While the main focus of the research that is motivated by problems in statistical physics has been on determining the critical threshold at which a random initial infected set percolates, fundamental extremal problems such as determining the size of the smallest percolating sets have  been investigated extensively as well. Indeed this problem is often closely related to the problem of determining the critical percolation threshold~\cite{MR2915649,MR2726074,MR2925568,MR2214907}. We will denote the size of the smallest percolating set in a graph $G$ in the $r$-neighbour bootstrap percolation process by $m(G,r)$.

In this paper, we are interested in a closely related bootstrap process, which we refer to as the \emph{$r$-bond bootstrap percolation}.  In this process, we start with a set of infected \emph{edges}, and at every step, the infection spreads to a new edge if at least one of its endpoints is incident to at least $r$ infected edges. In other words, once a vertex is incident to $r$ infected edges, then the infection spreads to all of the edges that are incident to that vertex. We denote the size of the smallest percolating set for this process by $m_e(G,r)$. This natural process seems to have been introduced first in~\cite{imbibition} for the two dimensional grid to model how a wetting fluid fills the ducts in the network of a porous media.   

The  $r$-bond bootstrap percolation  is  an instance of the graph bootstrap process defined in 1968 by Bollob\'as~\cite{Bollobas68}. Given graphs $G$ and $H$,  and an initial set of infected edges, in the $H$-bootstrap process, at each time step, we infect an edge $e$ if it completes a new infected copy of $H$ in $G$. Note that taking $H$ to be the star with $r+1$ leaves (denoted by $S_{r+1}$), results in the above-mentioned process.  The size of the smallest percolating set of edges in $G$ in the $H$-percolation process is called the weak saturation number of $H$ in $G$ and is denoted by $\wsat(G,H)$.  Hence in our notation $m_e(G,r) = \wsat(G,S_{r+1})$.

Note that one can turn a percolating set of vertices for the $r$-neighbour bootstrap process to a percolating set of edges for the $r$-bond bootstrap process by infecting $r$ arbitrarily chosen edges incident to every initially infected vertex  (if the degree of the vertex is less than $r$, then we just infect all the edges incident to it). Similarly, given a percolating set of edges for the $r$-bond bootstrap process, to obtain a percolating set for the $r$-neighbour bootstrap process, one can pick one endpoint of every infected edge; these vertices together with all the vertices of degree less than $r$ form  a percolating set of vertices for the $r$-neighbour bootstrap process. These observations show
\begin{equation}
\label{eq:simple_lowerbd}
\frac{m_e(G,r)}{r} \le m(G,r) \le m_e(G,r)+|\{v:\deg_G(v) < r\}|.
\end{equation}
Recently Morrison and Noel~\cite{MorrisonNoel}  used \eqref{eq:simple_lowerbd} to determine the asymptotics of $m(Q_d,r)$, where $Q_d$ denotes the $d$-dimensional hypercube. Indeed, they proved the exact formula 
\begin{equation}
\label{eq:exact_formula}
m_e(Q_d,r)= \sum_{j=1}^{r} {d-j-1 \choose r-j} j 2^{j-1},
\end{equation}
and combined it with \eqref{eq:simple_lowerbd} to show that $m(Q_d,r) = \frac{d^{r-1}}{r!} + \Theta_{d \to \infty}(d^{r-2})$, settling a conjecture of~\cite{MR2214907}. Prior to~\cite{MorrisonNoel}, the best known lower bounds for $m(Q_d,r)$  were only linear in $d$. 

%

The purpose of this article is to introduce a  simple \b{method} based on \b{the fact that a non-zero polynomial of degree $r$ has at most $r$ roots} for proving lower bounds for $m_e(G,r)$. We will use this method to settle a problem of Morrison and Noel~\cite{MorrisonNoel} by determining $m_e(G,r)$ for the multidimensional tori \b{(Theorem~\ref{thm:tori})}. Moreover,  we provide an alternative and simpler proof  for the case of the hypercube, and more generally, the multidimensional  grid \b{(Theorem~\ref{thm:grid})}, which were originally established in~\cite{MorrisonNoel}.

\b{Applying this property of polynomials to percolation seems very natural, and indeed  as an anonymous referee pointed  out to us,    Balister et al.~\cite{Balister2018LineP} have used  it recently to determine the size of the minimal percolating set of $d$-dimensional integer grid in a slightly different context of the line percolation model.  }

\subsection{Notation}
For a positive integer $n$, we denote $[n]=\{1,\ldots,n\}$. For a graph $G=(V,E)$, and an edge colouring $c:E \to \mathbb{R}$, to simplify the notation we often denote the colour of an edge $e$ by $c_e$. An edge colouring is called \emph{proper} if it assigns different colours to incident edges. For a logical statement $P$, we define $1_{[P]}$ to be $1$ if $P$ is true, and $0$ if $P$ is false. 

The Cartesian product of two graphs $G$ and $H$, denoted by $G \square H$, is the graph with vertex set $V(G) \times V(H)$, in which two vertices $(u_1,v_1)$ and $(u_2,v_2)$ are adjacent if and only if either $u_1=u_2$ and $v_1v_2 \in E(H)$, or $v_1=v_2$ and $u_1u_2 \in E(G)$. In other words, for every vertex $v \in V(H)$, we have a copy $G_v$ of $G$,  induced on vertices $\{(u,v): u \in V(G)\}$, and for every \b{edge} $v_1v_2 \in E(H)$, there is a  matching between $G_{v_1}$ and $G_{v_2}$ that connects each vertex of $G_{v_1}$ to its corresponding vertex in $G_{v_2}$.

\section{Polynomials and Bootstrap Percolation}
We start with the following key definition. 


\begin{definition}
\label{def:Wspace}

Let \b{$r \ge 0$}  be an  integer,  $G=(V,E)$ be a graph, and let  $c:E \to \mathbb{R}$ be a proper edge colouring of $G$. Let $W_{G,c}^r$ be the vector space of all vectors $(p_v \in \mathbb{R}[x]: v \in V)$, where $p_v$'s are univariate polynomials such that
 \begin{enumerate}
 \item \b{ $\deg(p_v)\le  \min \{r, \deg(v)\} - 1$ (here we use the convention that the degree of the zero polynomial is   $-1$);}
 \item  $p_{u}(c_{uv})=p_{v}(c_{uv})$ for every edge $uv \in E$.
\end{enumerate}

\end{definition}

Note that $W^r_{G,c}$ is indeed  a vector space as a  set of polynomials satisfying the above conditions is closed under addition and multiplication by scalars.

The following theorem summarizes the main idea of this article. 
\b{\begin{theorem} 
\label{thm:dim_Lowerbound}
Let  $c:E \to \mathbb{R}$ be a proper edge colouring of a graph $G=(V,E)$, and $r \ge 0$ be an integer. We have
$$m_e(G,r) \ge \dim(W^r_{G,c}).$$
\end{theorem} }
\begin{proof} 
\b{
Let $F \subseteq E$ be a percolating set for the $r$-bond bootstrap process in $G$.  
We claim that if a vector $(p_v)_{v \in V} \in W^r_{G,c}$ satisfies $p_u(c_{uv})=p_v(c_{uv})=0$ for all $uv \in F$, then $p_v \equiv 0$ for all $v \in V$.  To prove this claim, observe that throughout the process, the condition $p_u(c_{uv})=p_v(c_{uv})=0$ will be forced for the newly infected edges $uv$.  Note that if a vertex $u$ is incident to at least $\alpha_u := \min \{r, \deg(u)\}$  infected edges, then we know that $p_u$ has at least $\alpha_{u}$ distinct roots, as it has to evaluate to $0$ on the colours of its neighbouring infected edges. However, since the degree of $p_u$ is at most $\alpha_{u}-1$, this implies $p_u\equiv 0$, and thus $p_u$ evaluates to $0$ on all the edges incident to $u$. This corresponds to the  spreading of infection  to all the edges incident to $u$. Since $F$ percolates, eventually all the polynomials $p_u$ in the vector $(p_u)_{u \in V}$ will be forced to be equivalent to $0$. }

\b{
Now let us show that theorem follows from the above claim. Recall that 
$\alpha_v = \min \{r, \deg(v)\}$ for $v \in V$, and define the vector space  $X= \{ (p_u)_{u \in V} \ | \ \deg(p_u) \leq \alpha_{u}-1 \ \forall u \in V\}$, and its subspace $Y =  \big\{(p_u)_{u \in V} \in X  \ | \ p_u(c_{uv})=0 \ \forall uv \in F \big\} $. Note that in the definition of $Y$, we force the condition $p_u(c_{uv})=0$ only for one endpoint of the edge, which can be chosen arbitrarily. Clearly, $W_{G,c}^r \subset X$ and $Y \subset X$, and $\codim(Y) \le |F|$ in $X$.  The claim proven in the previous paragraph shows that  $W^r_{G,c} \cap  Y = \{\vec{0}\}.$ Hence  $$\dim(W^r_{G,c}) + \dim\left( Y \right) \leq \dim\left( X\right),$$
which yields the desired bound $\dim(W_{G,c}^r) \leq |F|$. }
\end{proof}

To warm up let us consider a simple example. Let $G=(V,E)$ be a graph with \emph{maximum} degree $r$. Obviously, $m_e(G,r)=|E|$, as in such a graph, the initial infection cannot spread to any new edges. The following proposition shows that the lower bound provided by Theorem~\ref{thm:dim_Lowerbound} is sharp for such graphs. 

\b{
\begin{proposition}
\label{prop:lowdegree}
Let $r$ be a non-negative integer, and let  $c:E \to \mathbb{R}$ be a proper edge colouring of a graph $G=(V,E)$  with maximum degree $r$. We have $m_e(G,r)=  \dim(W^r_{G,c})=|E|$. 
\end{proposition} }
\begin{proof}
 
Consider an edge $e_0=u_0v_0 \in E$. Let the vector $\mathbf{p}_{e_0} = (p_u)_{u \in V}$ for the  edge $e_0$ be defined as follows. 
Let $p_{u_0}$ be a polynomial of degree \b{$\deg(u_0)-1$} that is equal to $1$ on $c_{u_0v_0}$, and is equal to $0$ on $c_{u_0v}$ for every $u_0v \in E \setminus \{u_0v_0\}$.  \b{Such a polynomial exists since  $p_{u_0}$ has $\deg(u_0)-1$ roots, and this does not exceed its degree. } Similarly,  choose $p_{v_0}$ to be a polynomial of degree \b{$\deg(v_0)-1$} that is equal to $1$ on $c_{u_0v_0}$ and is equal to $0$ on $c_{uv_0}$ for every $uv_0 \in E \setminus \{u_0v_0\}$. Set $p_u \equiv 0$ for all $u \not\in \{u_0,v_0\}$. \v{Clearly, the vector $\mathbf{p}_{e_0}$, constructed in this way, belongs to  $W_{G,c}^r$. } Next note that the vectors $\{\mathbf{p}_{e_0} : e_0 \in E\}$ are linearly independent, and thus $\dim(W^r_{G,c}) \geq |E|$.   Combining this with Theorem~\ref{thm:dim_Lowerbound} and the trivial inequality $m_e(G,r) \leq |E|$, we deduce the statement of the proposition.
\end{proof}

\section{Tori and Grids}

In this section, we apply Theorem~\ref{thm:dim_Lowerbound} \b{to obtain recursive formulas for} $m_e(G,r)$ for tori and grids in arbitrary dimensions.  In fact, our results are more general as they apply to the Cartesian products of arbitrary graphs with cycles and paths. The cases of the tori and grids will follow easily from those    by simple inductions.  First we prove an upper bound on $m_e(G\square C_k,r)$ by constructing a percolating set of appropriate size. \b{Here we assume $r>1$, since, trivially, $m_e(G,0)=0$ and $m_e(G,1)=1$ for every connected graph $G$. }

\begin{proposition}
\label{prop:cycle_up}
Let \b{$r>1$}, $k \ge 3$ be integers, and $G=(V,E)$ be a graph. We have
$$m_e(G\square C_k,r)\leq m_e(G,r)+(k-2)m_e(G,r-1)+m_e(G,r-2)+d_{r-1}+k(d_0+\ldots+d_{r-2}),$$
where $d_t$ denotes the number of vertices with degree exactly $t$ in $G$.
\end{proposition}
\begin{proof}
For every vertex $v\in G$, denote its corresponding vertices in $G\square C_k$ by $v_1,v_2,\dots,v_k$. Let $G_1,\ldots,G_k$ denote the $k$ copies of $G$ in $G \square C_k$ corresponding to the $k$ vertices of $C_k$. First let us consider the case where every  vertex of $G$ is of degree at least $r$.  Construct a percolating set $F$ for $G \square C_k$ in the following manner. Pick an optimal $r$-percolating set $F_r(G_1)$ for $G_1$, optimal $(r-1)$-percolating sets $F_{r-1}(G_l)$ for $G_l$, $l \in \{2,\ldots,{k-1}\}$, and an optimal $(r-2)$-percolating set $F_{r-2}(G_k)$ for $G_k$. 

Since  $F_r(G_1) \subseteq F$, after running the $r$-bond  bootstrap process on $G_1$, all the edges in $G_1$ will be infected, and then due to the degree condition, the infection will pass to all the edges between $G_1$ and $G_2$, and  $G_1$ and $G_k$. Now every vertex in $G_2$ has an infected edge coming from $G_1$. This together with the edges in $F_{r-1}(G_2) \subseteq F$ infects all the edges in $G_2$, and consequently all the edges between $G_2$ and $G_3$ will be infected. Continuing in this manner, all the edges will be infected except possibly the edges inside $G_k$. However, at this point, every vertex in $G_k$ has two external infected edges incident to it,  one from $G_1$ and one from $G_{k-1}$. Thus the edges in $F_{r-2}(G_k) \subseteq F$ will eventually infect all the edges in $G_k$.

It remains to deal with the vertices of degrees less than $r$. If $\deg_G(v) = r-1$, then we only need to add the edge $v_1v_k$ to the above set. This will guarantee that once  $G_1$ is fully infected, $v_1v_2$ will become infected, and the process proceeds as it is described above. Finally,  for the vertices with degree $\deg_G(v) < r-1$, one can (and must) simply include  all the edges $v_iv_{i+1}$ for $i = 1,\ldots,k$ (let $v_{k+1} = v_1$). 

\end{proof}

Now we turn to proving a lower bound for $m_e(G\square C_k,r)$. By Theorem~\ref{thm:dim_Lowerbound}, it suffices  to prove a lower bound for  $\dim(W_{G\square C_k,c'}^r)$ where $c'$ is a proper edge colouring of $G \square C_k$. This is achieved in Theorem~\ref{thm:cycle}  below, which complements Proposition~\ref{prop:cycle_up}.

\b{
\begin{theorem}
\label{thm:cycle}
Let \b{$r>1$}, $k \ge 3$ be integers,   and  $c:E \to \mathbb{R}$ be a proper edge colouring of a graph $G=(V,E)$. There exists a proper edge colouring $c'$ of $G \square C_k$ for which 
$$\dim(W_{G\square C_k,c'}^r)\geq\dim(W_{G,c}^r)+(k-2)\dim(W_{G,c}^{r-1})+\dim(W_{G,c}^{r-2})+d_{r-1}+k(d_0+\ldots+d_{r-2}),$$
where $d_t$ denotes the number of vertices with degree exactly $t$ in $G$.
\end{theorem} }

\begin{proof}
\b{
For every vertex $v\in G$, denote its corresponding vertices in $G\square C_k$ by $v_1,v_2,\dots,v_k$. Let $G_1,\ldots,G_k$ denote the $k$ copies of $G$ in $G \square C_k$ corresponding to the $k$ vertices of $C_k$. Let $\alpha_1,\ldots,\alpha_k$ be distinct real numbers that do not belong to $c(E)$. Let $c'$ be the proper colouring of $G \square C_k$ that is consistent with $c$ on $G_1,\ldots,G_k$ and moreover $c'(v_iv_{i+1})=\alpha_i$ for all $i \in [k]$ and $v \in V(G)$ (where $v_{k+1}=v_1$). To prove the theorem, we are going to find $\dim(W_{G,c}^r)+(k-2)\dim(W_{G,c}^{r-1})+\dim(W_{G,c}^{r-2})+d_{r-1}+k(d_0+\ldots+d_{r-2})$ linearly independent vectors in $W_{G\square C_k,c'}^r$.}

\b{Consider a linear basis $B^{(r)}$ for $W_{G,c}^r$. Pick a vector $\mathbf{q} \in B^{(r)}$ and let $\mathbf{q} =(q_v)_{v\in V(G)}$. Define the vector $\mathbf{p} ^{(1)}_\mathbf{q}  = (p_u)_{u\in V(G\square C_k)}$ as $p_{v_i}=q_v$ for $i\in[k]$ and $v \in V(G)$. (See Figure~\ref{fig:case1}). Trivially, the two conditions in Definition~\ref{def:Wspace} are satisfied, and thus the vector $\mathbf{p} ^{(1)}_\mathbf{q}$ belongs to $W_{G\square C_k,c'}^r$. Set $B_1=\{\mathbf{p} ^{(1)}_\mathbf{q} :  \mathbf{q} \in B^{(r)}\}$. Note that the restriction of $\mathbf{p} ^{(1)}_\mathbf{q}$ to $G_1$ equals $\mathbf{q}$, and thus the vectors in $B_1$ are linearly independent. }

\begin{figure}[ht]
	\begin{center}
\begin{tikzpicture}[scale=1.5,line cap=round,line join=round,>=triangle 45,x=1cm,y=1cm];
\draw [rotate around={90:(-5,0)},line width=2pt] (-5,0) ellipse (1.2640130858308254cm and 0.7731294077653242cm);
\draw [rotate around={90:(-2.5873915716139604,0)},line width=2pt] (-2.5873915716139604,0) ellipse (1.276829331251797cm and 0.7939100334073831cm);
\draw [rotate around={90:(1.409808073219042,0)},line width=2pt] (1.409808073219042,0) ellipse (1.2653040967520912cm and 0.7752383228773105cm);\draw [line width=1pt] (-5.006104774264048,0.6108243972406264)-- (-2.606356755426423,0.6081713697693732);
\draw [line width=1pt] (-2.606356755426423,0.6081713697693732)-- (-1.4680950754729831,0.6108243972406264);
\draw [line width=1.6pt,dashed] (-0.8226191364564878,0)-- (-0.2581452064027107,0);
\draw [line width=1pt] (1.4014081479553946,0.6081713697693732)-- (0.19343393764031155,0.6081713697693732);
\draw [shift={(-1.80745680222131,-4.622752092645986)},line width=1pt]  plot[domain=1.0205520321456385:2.119394590302861,variable=\t]({1*6.136723436665533*cos(\t r)+0*6.136723436665533*sin(\t r)},{0*6.136723436665533*cos(\t r)+1*6.136723436665533*sin(\t r)});
\draw (-5.194672627773725,-0.7187744153330075) node[anchor=north west] {$\mathit{G_1}$};
\draw (-2.730293614151421,-0.7187744153330075) node[anchor=north west] {$\mathit{G_2}$};
\draw (1.2871575903474395,-0.7187744153330075) node[anchor=north west] {$\mathit{G_k}$};
\draw (-3.821294739973795,0.85) node[anchor=north west] {$\mathit{\alpha_1}$};
\draw (-1.7933161766804406,0.85) node[anchor=north west] {$\mathit{\alpha_2}$};
\draw (0.1,0.85) node[anchor=north west] {$\mathit{\alpha_{k-1}}$};
\draw (-1.8959986355813698,1.74) node[anchor=north west] {$\mathit{\alpha_k}$};
\draw (-5.22,0.92) node[anchor=north west] {$\mathit{v_1}$};
\draw (-2.672534731019648,0.9) node[anchor=north west] {$\mathit{v_2}$};
\draw (1.3705870882044446,0.9) node[anchor=north west] {$\mathit{v_k}$};
\draw (-5.25,0.58) node[anchor=north west] {$p_{v_1}=q_v$};
\draw (-2.85,0.58) node[anchor=north west] {$p_{v_2}=q_v$};
\draw (1.15,0.58) node[anchor=north west] {$p_{v_k}=q_v$};\begin{scriptsize}
\draw [fill=black] (-5.006104774264048,0.6108243972406264) circle (1.5pt);\draw [fill=black] (-2.606356755426423,0.6081713697693732) circle (1.5pt);\draw [fill=black] (1.4014081479553946,0.6081713697693732) circle (1.5pt);\end{scriptsize}
\end{tikzpicture}
	\end{center}
	\caption{\label{fig:case1} The vector $(q_v)_{v\in V(G)} \in B^{(r)}$ is used to construct $\mathbf{p}^{(1)}_\mathbf{q} = (p_u)_{u\in V(G\square C_k)}$ such that $\mathbf{p}^{(1)}_\mathbf{q} \in B_1 \subseteq W_{G\square C_k,c'}^r$.}
\end{figure}




\b{ Next, consider a linear basis $B^{(r-1)}$ for $W_{G,c}^{r-1}$, and fix $\ell \in \{2,\ldots,k-1\}$. Let $\mathbf{q} \in B^{(r-1)}$ and $\mathbf{q}=(q_v)_{v\in V(G)}$.  Note that $\deg(q_v) \le \min\{r-1, \deg_G(v)\}-1$  for all $v$. Define the vectors $\mathbf{p}^{(\ell)}_\mathbf{q}=(p_u)_{u\in V(G\square C_k)}$ as 
$$p_{v_\ell}=\frac{x-\alpha_{\ell-1}}{\alpha_\ell-\alpha_{\ell-1}}q_v, \qquad p_{v_{\ell+1}}=\frac{x-\alpha_{\ell+1}}{\alpha_\ell-\alpha_{\ell+1}}q_v, \qquad
\text{for all } v \in V(G),$$ 
and  $p_{v_j} \equiv 0$ for all $j \not\in \{\ell,\ell+1\}$ and $v \in V(G)$. (See Figure~\ref{fig:case2}). Note that $\deg(p_{v_\ell})  \leq  \deg(q_v)+1 = \min \{r-1, \deg_G(v)\} = \min \{r-1, \deg(v_\ell) - 2\} \leq \min \{ r, \deg(v_\ell) \} -1$, and similarly, $\deg(p_{v_{\ell +1}}) \leq \min \{r,\deg(v_{\ell+1})\} -1$. Moreover, $p_{v_\ell}(\alpha_{\ell})=p_{v_{\ell+1}}(\alpha_{\ell})=q_v(\alpha_\ell)$, and $p_{v_\ell}(\alpha_{\ell-1})=p_{v_{\ell+1}}(\alpha_{\ell+1})=0$. Hence both conditions in  Definition~\ref{def:Wspace} are satisfied, and the vectors $\mathbf{p}^{(\ell)}_\mathbf{q}$ belong to $W_{G\square C_k,c'}^r$. Define $B_\ell=\{ \mathbf{p}^{(\ell)}_\mathbf{q} : \mathbf{q} \in B^{(r-1)}\}$. Note that the restriction of $\mathbf{p}^{(\ell)}_\mathbf{q}$ to $G_\ell$ is the following vector $\Big( \frac{x-\alpha_{\ell-1}}{\alpha_\ell-\alpha_{\ell-1}} q_v \Big)_{v \in V(G)}$, and since $\frac{x-\alpha_{\ell-1}}{\alpha_\ell-\alpha_{\ell-1}} \not\equiv 0$,  the vectors in $B_\ell$ are linearly independent. }

\begin{figure}[ht]
	\begin{center}
\begin{tikzpicture}[scale=1.5, line cap=round,line join=round,>=triangle 45,x=1cm,y=1cm]
\draw [rotate around={90:(-5.804939824256687,-0.01603105961352884)},line width=2pt] (-5.804939824256687,-0.01603105961352884) ellipse (1.2640130858308207cm and 0.7731294077653214cm);
\draw [rotate around={90:(-2,0)},line width=2pt] (-2,0) ellipse (1.2768293312517978cm and 0.7939100334073836cm);
\draw [rotate around={90:(3.8093286203523364,0.019143388232187277)},line width=2pt] (3.8093286203523364,0.019143388232187277) ellipse (1.2653040967520905cm and 0.7752383228773104cm);
\draw [line width=1pt] (-2.0043453090240657,0.5900743355718491)-- (-0.011213519910090805,0.6013278931046495);
\draw [line width=1.6pt,dashed] (1.6081185011410888,0)-- (2.2064608669980923,0);
\draw [line width=1pt] (3.8012690615790294,0.6037346046791505)-- (2.585505111029205,0.6037346046791509);
\draw [shift={(-0.9845379188184807,-6.874511416767466)},line width=1pt]  plot[domain=1.001508794470347:2.140950107832702,variable=\t]({1*8.878519697049937*cos(\t r)+0*8.878519697049937*sin(\t r)},{0*8.878519697049937*cos(\t r)+1*8.878519697049937*sin(\t r)});
\draw [rotate around={90:(0,0)},line width=2pt] (0,0) ellipse (1.2569887070810324cm and 0.7615908414163378cm);
\draw [line width=1pt] (-0.011213519910090805,0.6013278931046495)-- (1.395198934724808,0.6043555260022104);
\draw [line width=1pt] (-2.0043453090240657,0.5900743355718491)-- (-3.3983136792918662,0.5919370995410274);
\draw [line width=1pt] (-5.788239851746554,0.6173948737864519)-- (-4.613456708518633,0.6173948737864522);
\draw [line width=1.6pt,dashed] (-4.18998836619229,0.0026827639578891546)-- (-3.616257063685631,0.0026827639578891546);
\draw (-5.987863491319479,-0.7364753211228021) node[anchor=north west] {$\mathit{G_1}$};
\draw (-2.1833092027570236,-0.7364753211228021) node[anchor=north west] {$\mathit{G_{\ell}}$};
\draw (-0.25280836197310713,-0.7364753211228021) node[anchor=north west] {$\mathit{G_{\ell+1}}$};
\draw (3.6194827981958015,-0.7364753211228021) node[anchor=north west] {$\mathit{G_k}$};
\draw (-5.987863491319479,0.92) node[anchor=north west] {$\mathit{v_1}$};
\draw (-2.149440766953797,0.90) node[anchor=north west] {$\mathit{v_{\ell}}$};
\draw (-0.12862409736127625,0.9) node[anchor=north west] {$\mathit{v_{\ell+1}}$};
\draw (3.709798627004406,0.9) node[anchor=north west] {$\mathit{v_k}$};
\draw (-1.1333876928569988,2.3) node[anchor=north west] {$\mathit{\alpha_k}$};
\draw (-5.107284160435587,0.85) node[anchor=north west] {$\mathit{\alpha_1}$};
\draw (-3.3122570628645773,0.85) node[anchor=north west] {$\alpha_{\ell-1}$};
\draw (-1.1108087356548477,0.85) node[anchor=north west] {$\mathit{\alpha_{\ell}}$};
\draw (0.7067973191183133,0.85) node[anchor=north west] {$\alpha_{\ell+1}$};
\draw (2.49,0.85) node[anchor=north west] {$\mathit{\alpha_{k-1}}$};
\draw (-6.05,0.63) node[anchor=north west] {$\mathit{p_{v_1}=0}$};
\draw (3.55,0.63) node[anchor=north west] {$\mathit{p_{v_k}=0}$};
\draw (-4.5,-0.5558436635055932) node[anchor=north west] {$\mathit{p_{v_{\ell}}=\frac{x-\alpha_{\ell-1}}{\alpha_{\ell}-\alpha_{\ell-1}}q_v}$};\draw (0.7632447121236912,-0.5558436635055932) node[anchor=north west] {$\mathit{p_{v_{\ell+1}}=\frac{x-\alpha_{\ell+1}}{\alpha_{\ell}-\alpha_{\ell+1}}q_v}$};\draw [->,line width=0.2pt] (-3.2558096698592,-0.6235805351120465)-- (-2.0704144167462686,0.43763045338905415);\draw [->,line width=0.2pt] (1.2599817705710148,-0.6235805351120465)-- (0.1,0.44);
\begin{scriptsize}\draw [fill=black] (-5.788239851746554,0.6173948737864519) circle (1.5pt);\draw [fill=black] (-2.0043453090240657,0.5900743355718491) circle (1.5pt);\draw [fill=black] (3.8012690615790294,0.6037346046791505) circle (1.5pt);\draw [fill=black] (-0.011213519910090805,0.6013278931046495) circle (1.5pt);\end{scriptsize}
\end{tikzpicture}

	\end{center}
	\caption{\label{fig:case2} The vector $(q_v)_{v\in V(G)} \in B^{(r-1)}$ is used to construct $\mathbf{p}^{(\ell)}_\mathbf{q} = (p_u)_{u\in V(G\square C_k)}$ such that  $\mathbf{p}^{(\ell)}_\mathbf{q} \in B_\ell \subseteq W_{G\square C_k,c'}^r$.}
\end{figure}

\b{
Finally, consider a linear basis $B^{(r-2)}$ for $W_{G,c}^{r-2}$, and pick $\mathbf{q} \in B^{(r-2)}$ and let $\mathbf{q} = (q_v)_{v\in G}$. Note that  $\deg(q_v) \leq \min \{r-2, \deg_G(v)\}-1$.}

\b{ Define the vector $\mathbf{p}^{(k)}_\mathbf{q} = (p_u)_{u\in V(G\square C_k)}$ as $p_{v_i}\equiv0$ for all $i\leq k-1$, and $p_{v_k}=(x-\alpha_{k-1})(x-\alpha_k)q_v$. (See Figure~\ref{fig:case3}). Note that  $\deg(p_{v_k})  \leq 2+\deg(q_v) = \min \{r-2, \deg(v_k)-2\}+1= \min \{\deg(v_k), r\} - 1$, and $p_{v_k}(\alpha_{k})=p_{v_k}(\alpha_{k-1})=0$. Hence the vector $\mathbf{p}^{(k)}_\mathbf{q}$ belongs to $W_{G\square C_k,c'}^r$. Define $B_k=\{\mathbf{p}^{(k)}_\mathbf{q} : \mathbf{q} \in B^{(r-2)}\}$. Similar to above, the restriction of $\mathbf{p}^{(k)}_\mathbf{q}$ to $G_k$ shows that the vectors in $B_k$ are linearly independent.  }
 
\begin{figure}[ht]
	\begin{center}
\begin{tikzpicture}[scale=1.5,line cap=round,line join=round,>=triangle 45,x=1cm,y=1cm]
\draw [rotate around={90:(-5,0)},line width=2pt] (-5,0) ellipse (1.2640130858308254cm and 0.7731294077653242cm);
\draw [rotate around={90:(-2.5873915716139604,0)},line width=2pt] (-2.5873915716139604,0) ellipse (1.276829331251797cm and 0.7939100334073831cm);\draw [rotate around={90:(1.409808073219042,0)},line width=2pt] (1.409808073219042,0) ellipse (1.2653040967520912cm and 0.7752383228773105cm);\draw [line width=0.8pt] (-5.006104774264048,0.6108243972406264)-- (-2.606356755426423,0.6081713697693732);
\draw [line width=0.8pt] (-2.606356755426423,0.6081713697693732)-- (-1.4680950754729831,0.6108243972406264);
\draw [line width=2pt,dashed] (-0.8226191364564878,0)-- (-0.2581452064027107,0);
\draw [line width=0.8pt] (1.4014081479553946,0.6081713697693732)-- (0.19343393764031155,0.6081713697693732);
\draw [shift={(-1.80745680222131,-4.622752092645986)},line width=0.8pt]  plot[domain=1.0205520321456385:2.119394590302861,variable=\t]({1*6.136723436665533*cos(\t r)+0*6.136723436665533*sin(\t r)},{0*6.136723436665533*cos(\t r)+1*6.136723436665533*sin(\t r)});
\draw (-5.180874700005004,-0.7850218873788958) node[anchor=north west] {$\mathit{G_1}$};
\draw (-2.7763623745811246,-0.8143452084206504) node[anchor=north west] {$\mathit{G_2}$};
\draw (1.2409326081392609,-0.8045707680733989) node[anchor=north west] {$\mathit{G_k}$};
\draw (-5.2,0.92) node[anchor=north west] {$\mathit{v_1}$};
\draw (-2.7,0.9) node[anchor=north west] {$\mathit{v_2}$};
\draw (1.3,0.9) node[anchor=north west] {$\mathit{v_k}$};
\draw (-3.82222749173704,0.88) node[anchor=north west] {$\mathit{\alpha_1}$};
\draw (-1.8282416608977246,0.88) node[anchor=north west] {$\mathit{\alpha_2}$};
\draw (0.1,0.88) node[anchor=north west] {$\mathit{\alpha_{k-1}}$};
\draw (-1.945534945064743,1.8) node[anchor=north west] {$\mathit{\alpha_k}$};
\draw (-5.3,0.63) node[anchor=north west] {$\mathit{p_{v_1}=0}$};
\draw (-2.9,0.63) node[anchor=north west] {$\mathit{p_{v_2}=0}$};
\draw (2.130426228619845,1.2673075778931697) node[anchor=north west] {$p_{v_k}=(x-\alpha_{k-1})(x-\alpha_k)q_v$};\draw [<-, line width=0.2pt] (1.5964338907889724,0.6215494019116472)-- (2.3042841990764087,0.8699179311353097);\begin{scriptsize}\draw [fill=black] (-5.006104774264048,0.6108243972406264) circle (1.5pt);\draw [fill=black] (-2.606356755426423,0.6081713697693732) circle (1.5pt);\draw [fill=black] (1.4014081479553946,0.6081713697693732) circle (1.5pt);\end{scriptsize}
\end{tikzpicture}
	\end{center}
	\caption{\label{fig:case3} The vector $(q_v)_{v\in V(G)}\in B^{(r-2)}$ are used to construct $\mathbf{p}^{(k)}_\mathbf{q} = (p_u)_{u\in V(G\square C_k)}$ such that  $\mathbf{p}^{(k)}_\mathbf{q} \in B_k \subseteq W_{G\square C_k,c'}^r$.}
\end{figure}

\b{
We will show that the elements of $B_1 \cup \ldots \cup B_k$ are linearly independent.  We have already shown that the vectors in each individual $B_i$ are linearly independent.  Next, note that if a vector $(p_u)_{u \in V(G\square C_k)} \in B_j$ for some $j \in [k]$, then $p_v \equiv 0$ for all $v \in \bigcup_{i=1}^{j-1} V(G_i)$, and $p_v \not \equiv 0$ for at least one vertex $v \in V(G_j)$.  
 Thus $B_j$ does not intersect the span of $B_{j+1} \cup \ldots \cup B_{k}$.  These show that $B_1 \cup \ldots \cup B_k$  consists  of exactly $\dim(W_{G,c}^r)+(k-2)\dim(W_{G,c}^{r-1})+\dim(W_{G,c}^{r-2})$ linearly independent vectors.  }

\b{To achieve the desired lower bound, we need to extend $B_1 \cup \ldots \cup B_k$ to include an additional  set of   $d_{r-1}+k(d_0+\ldots+d_{r-2})$  linearly independent vectors. First we show that no non-zero vector in $\mathrm{span}(B_1 \cup \ldots \cup B_k)$ evaluates to $0$ on the colours of all the edges in $E(G_1) \cup \ldots \cup E(G_k)$. Indeed consider a non-zero $\mathbf{p} = (p_u)_{u \in V(G \square C_k)}$ in the span of $B_1 \cup \ldots \cup B_k$, and consider the smallest $j$ such that there exists $u_j \in G_j$ with $p_{u_j} \not\equiv 0$. If $j=1$, then $p_{u_j}=q_u$ for  some $\mathbf{q}=(q_u)_{u\in V(G)} \in \span(B^{(r)})$, if $j \in [2,k-1]$, then $p_{u_j}=\frac{x-\alpha_{\ell-1}}{\alpha_\ell-\alpha_{\ell-1}}q_u$ for some  $\mathbf{q} \in \span(B^{(r-1)})$, and if $j=k$, then $p_{u_j}=(x-\alpha_{k-1})(x-\alpha_k)q_u$ for some $\mathbf{q} \in \span(B^{(r-2)})$. In the first case $\deg(q_u)\le \min \{r, \deg(u)\}-1$, in the second case $\deg(q_u)\le \min \{r-1, \deg(u)\}-1$, and in the third case $\deg(q_u)\le \min \{r-2, \deg(u)\}-1$. Due to these degree restrictions, in neither of these cases $q_u$ can be a non-zero polynomial that evaluates to $0$ on   the colours of   all the edges incident to it in $G$. }
 
 \b{Hence to finish the proof, it suffices to find   $d_{r-1}+k(d_0+\ldots+d_{r-2})$ linearly independent vectors in $W_{G\square C_k,c'}^r$  such that they all evaluate to $0$ on  the colours of  $E(G_1) \cup \ldots \cup E(G_k)$. This together with the previous paragraph  will guarantee that these vectors are independent from $B_1 \cup \ldots \cup B_k$. }

\b{For every vertex $v$ in $G$ with $\deg_G(v)=r-1$, let $\mathbf{p}_v = (p_u)_{u \in V(G \square C_k)}$ be the vector defined as  
$$p_{v_1}(x)=\ldots=p_{v_k}(x) := \prod_{u:vu \in E(G)}(x-c_{vu}),$$ and $p_{u_1}=\ldots=p_{u_k} \equiv 0$ for all $u \neq v$.   Note, for $i \in [k]$, $\deg(p_{v_i})= r-1 = \min\{{r, \deg(v_i)\}}-1$. So, it follows that $\mathbf{p}_v \in W_{G\square C_k,c'}^r$.}

\b{ Finally, for every vertex $v$  in $G$ with  $\deg_G(v) \le r-2$ and every $i_0 \in [k]$, define the vector $\mathbf{p}_{v}^{i_0}=(p_u)_{u \in V(G \square C_k)}$ as follows. Set $p_u \equiv 0$ for all $u \not\in \{v_{i_0}, v_{i_0+1} \}$.  Let $p_{v_{i_0}}$   be the unique polynomials of degree  $\deg(v_{i_0})-1= \deg_G(v)+1 \leq r-1$   that  is equal to $1$ on $c'_{v_{i_0} v_{i_0+1}}$, and is equal to $0$ on all the colours incident to $v_{i_0}$. Similarly, let $p_{v_{i_0+1}}$ be the unique polynomial of degree   $\deg(v_{i_0+1})-1= \deg_G(v)+1 \leq r-1$  that  is equal to $1$ on $c'_{v_{i_0} v_{i_0+1}}$, and is equal to $0$ on all the other colours incident to $v_{i_0+1}$.   Obviously,  $\mathbf{p}_{v}^{i_0}$ belongs to $W_{G\square C_k,c'}^r$, and moreover  the vectors 
$$\{\mathbf{p}_v: \deg_G(v)=r-1\} \cup \{\mathbf{p}_{v}^i: \deg_G(v)<r-1, i \in [k]\}$$ 
 are all linearly independent, and they evaluate to  $0$ on  the colours of  $E(G_1) \cup \ldots \cup E(G_k)$, and thus are independent from  $B_1 \cup \ldots \cup B_k$. Adding these  $d_{r-1}+k(d_0+\ldots+d_{r-2})$ vectors to the original $B_1 \cup \ldots \cup B_k$ yields a linearly independent set of the desired size in $W_{G\square C_k,c'}^r$.}

\end{proof}

Next we \b{state} the analogues of Proposition~\ref{prop:cycle_up} and Theorem~\ref{thm:cycle}  for the Cartesian product of arbitrary graphs with paths. 
\begin{proposition}
\label{prop:path_up}
Let $r>0$, $k \ge 2$ be integers, and $G=(V,E)$ be a graph. We have
$$m_e(G\square P_k,r)\leq m_e(G,r)+(k-1)m_e(G,r-1)+d_{r-1}+(k-1)(d_0+\ldots+d_{r-2}),$$
where $d_t$ denotes the number of vertices with degree exactly $t$ in $G$.
\end{proposition}
The proof of this proposition is almost identical to the proof of Proposition~\ref{prop:cycle_up}. The only difference is that instead of an $(r-2)$-percolating set, we pick an $(r-1)$-percolating set for $G_k$.

Similar to Theorem~\ref{thm:cycle} one can complement Proposition~\ref{prop:path_up} with the following lower bound. 
\begin{theorem}
\label{thm:path}
Let $r>0$, $k \ge 2$ be integers,  $G=(V,E)$ be a graph, and  $c:E \to \mathbb{R}$ be a proper edge colouring of $G$. There exists a proper edge colouring $c'$ of $G \square P_k$ for which 
$$\dim(W_{G\square \b{P_k},c'}^r)\geq\dim(W_{G,c}^r)+(k-1)\dim(W_{G,c}^{r-1})+d_{r-1}+(k-1)(d_0+\ldots+d_{r-2}),$$
where $d_t$ denotes the number of vertices with degree exactly $t$ in $G$.
\end{theorem}

The proof of this theorem is almost identical to the proof of Theorem~\ref{thm:cycle}. The only difference is that $B_k$ is also constructed similar to $B_2,\ldots,B_{k-1}$.

\subsection{Exact bounds for Grids and Tori}
With Theorems~\ref{thm:cycle}~and~\ref{thm:path} in hand, it is easy to prove a recursive formula for the sizes of the smallest percolating sets in tori and grids. We start with the tori.

\begin{theorem}
\label{thm:tori}
Let $d >  0$, $a_1,\ldots,a_d \ge 3$, and $r \ge 0$ be integers. Denoting $G_j = \square_{i=1}^j C_{a_i}$ for $j \in [d]$, we have
\begin{equation}
\label{eq:tori_equality}
m_e(G_d ,r)=m_e(G_{d-1},r) + (a_d-2) m_e(G_{d-1},r-1) +  m_e(G_{d-1},r-2) + 
\left\{ \begin{array}{lcl}
0 & \qquad & r < 2d-1 \\
\prod_{i=1}^{d-1}a_i & & r=2d-1 \\
\prod_{i=1}^{d}a_i & & r\ge 2d \\
\end{array}\right.,
\end{equation}
where  $G_0$ is the graph with a single vertex. 
\end{theorem}
\begin{proof}
Note that $G_{d-1}$ is a $2(d-1)$-regular graph, and thus Proposition~\ref{prop:cycle_up} implies that $m_e(G_d ,r)$ is bounded from above by  the right hand side of \eqref{eq:tori_equality}.  We will show that the other direction  follows from Theorem~\ref{thm:dim_Lowerbound}, Theorem~\ref{thm:cycle} and a simple induction with the base case $m_e(G_0,r')=\dim( W_{G_0,c}^{r'})=0$ for every $r'$. Indeed, assuming the existence of a coloring $c$ of $G_{d-1}$ with $m_e(G_{d-1},r')=\dim(W_{G_{d-1},c}^{r'})$ for every $r'$, one can use  Theorem~\ref{thm:cycle} to obtain a proper edge colouring $c'$ for $G_d$ with 
$$\dim(W_{G_d,c'}^r) \ge m_e(G_{d-1},r) + (a_d-2) m_e(G_{d-1},r-1) +  m_e(G_{d-1},r-2) + 
\left\{ \begin{array}{lcl}
0 & \qquad & r < 2d-1 \\
\prod_{i=1}^{d-1}a_i & & r=2d-1 \\
\prod_{i=1}^{d}a_i & & r\ge 2d \\
\end{array}\right..$$
This together with Theorem~\ref{thm:dim_Lowerbound} completes  the induction step and shows  
$$\dim(W_{G_d,c'}^r) =m_e(G_d ,r) = \mbox{R.H.S. of ($\ref{eq:tori_equality}$)}.$$
\end{proof}

The case of the multidimensional grid can be proven similar to Theorem~\ref{thm:tori}, however, since the \b{product of paths} is not a regular graph, the formula is  more complex. 

\begin{theorem}
\label{thm:grid}
Let $d  > 0$, $a_1,\ldots,a_d \ge 2$, and $r \ge 0$ be integers. Denoting $G_j = \square_{i=1}^j P_{a_i}$, we have
\begin{eqnarray}
\label{eq:path_equality}
m_e(G_d ,r) &=& m_e(G_{d-1},r) + (a_d-1) m_e(G_{d-1},r-1) +  \sum_{S \subseteq [d-1],  |S| =  r-d}  2^{d-1-|S|} \prod_{i \in S} (a_i-2)  \nonumber \\ 
&&+ a_d  \sum_{S \subseteq [d-1],  |S| <  r-d}  2^{d-1-|S|} \prod_{i \in S} (a_i-2) . 
\end{eqnarray}
where  $G_0$ is the graph with a single vertex. 
\end{theorem}
The proof of this theorem proceeds similar to the proof of Theorem~\ref{thm:tori}.
\b{Here} $\sum_{S \subseteq [d-1],  |S| =  r-d}  2^{d-1-|S|} \prod_{i \in S} (a_i-2)$ is the number of vertices of degree $r-1$ in $G_{d-1}$, and $\sum_{S \subseteq [d-1],  |S| <  r-d}  2^{d-1-|S|} \prod_{i \in S} (a_i-2)$ is the number of vertices of degree  less than $r-1$. \b{We leave this as an exercise to the interested reader.}

\subsection{\b{ Hypercubes}}
\b{While Theorem~\ref{thm:path} provides a recursive formula for the case of the grids, it is not clear whether there is a closed-form  solution to this recursive formula.  However for the special case of the hypercube, as it is shown in~\cite{MorrisonNoel}, it is possible to obtain a closed-form formula.  Let $d \ge r \ge 0$ be integers, and let $Q_d$ denote the $d$-dimensional hypercube. More formally the vertices of $Q_d$ are vectors $x \in \{0,1\}^d$ and two vertices are adjacent if they differ only in one coordinate. }

\b{Since $Q_d=P_2^{\square d}$, we have the following corollary to Theorem~\ref{thm:grid} that can be verified by a simple induction. This recovers the main result of~\cite{MorrisonNoel}.}

\begin{corollary}
\label{cor}
We have 
$$m_e(Q_d,r)= \sum_{j=1}^{r}{d -j -1 \choose r-j} j 2^{j-1}.$$ 
\end{corollary}

\b{It is also possible to explicitly describe a set of percolating edges that achieves this bound. The recursive construction discussed in the previous section is as follows.} Define $F_r(Q_d) \subseteq E(Q_d)$ 
in the following manner. If $r=0$, set $F_r(Q_d):=\emptyset$, and if $d=r$, set $F_r(Q_d):=E(Q_d)$. Otherwise let
\begin{equation}
\label{eq:recursive_struct}
F_r(Q_d) := F_r(Q'_{d-1}) \cup F_{r-1}(Q''_{d-1}),
\end{equation}
where $Q'_{d-1}$ and $Q''_{d-1}$ are the two copies of $Q_{d-1}$ in $Q_d$ induced on the vertices $x=(x_1,\ldots,x_d)$ with $x_d=1$ and  $x_d=0$, respectively.   A simple induction  \b{proves the explicit description}  
$$F_{r}(Q_d) = \left\{(x, \delta_j x) : j \in [n], \sum_{i=1}^{j-1} x_i \ge d-r\right\},$$
where $\delta_j  x$ is the vector that is obtained from $x$ by flipping the value of the $j$-th coordinate. Note 
$$|F_r(Q_d)|=  \sum_{j=1}^{r}{d -j -1 \choose r-j} j 2^{j-1}.$$

\section{Concluding remarks}
The polynomial method as it is used in Theorem~\ref{thm:dim_Lowerbound} is applicable to \b{a more general} setting. Let $H=(V,E)$ be a \emph{hypergraph}, and let $r$ be a nonnegative integer. Suppose that we initially infect a subset $F$ of the \emph{vertices}. We start a process in which, at every step, if there is a hyperedge $S \in E$ that contains at least $r$ infected vertices, then the infection spreads to all the vertices in $S$. \b{To prove a lower-bound for the size of the smallest percolating set for this percolation process we define a vector space similar to  Definition~\ref{def:Wspace}.}
\b{
\begin{definition}
Let $r$ be a positive integer,  $H=(V,E)$ be a hypergraph, and let  $c:V \to \mathbb{R}$ be a vertex colouring of $H$ such that it assigns distinct colors to the vertices in each hyperedge. Let $W_{H,c}^r$ be the vector space of all vectors $(p_e \in \mathbb{R}[x]: e \in E)$, where $p_e$'s are univariate polynomials such that
 \begin{enumerate}
 \item \b{$\deg(p_e)\le \ \min \{r, |e|\} - 1$  for all $e \in E$;}
 \item $p_{e_1}(c_{v})=p_{e_2}(c_{v})$  for every pair $e_1, e_2 \in E$ and for all $v \in e_1 \cap e_2$.
\end{enumerate}
\end{definition}
}

\b{It is not difficult to see that, similar to Theorem~\ref{thm:dim_Lowerbound}, the dimension of the vector space $W_{H,c}^r$ is a lower bound for the size of the smallest percolating set in $H$.}

\b{This is more general than Theorem~\ref{thm:dim_Lowerbound}, as given a graph $G$, to recover Theorem~\ref{thm:dim_Lowerbound},  it suffices to consider the hypergraph $H$ with $V(H):=E(G)$, and hyperedges $S_v=\{uv: uv \in E(G)\}$ for $v \in V(G)$. }

\b{Another well-studied class of extremal percolation problems that falls into this framework is the special case of the $\mathcal{H}$-bootstrap process\b{~\cite{MR2915649}}, when $\mathcal{H}$ is a $k$-uniform hypergraph. In the $\mathcal{H}$-bootstrap process,  we are given a  hypergraph $\mathcal{H}$, and an initial set of infected  vertices.   At each time step, we infect a vertex $u$ if it lies in an edge of $\mathcal{H}$ in which all vertices other than $u$ are already infected. Note that if $\mathcal{H}$ is a $k$-uniform hypergraph, then this process is equivalent to the process described above with  $r=k-1$.  In particular,  the graph bootstrap process of Bollob\'as~\cite{Bollobas68} that was mentioned \b{in the introduction} is a special case of this. Recall that given  graphs $G$ and $H$, and an initial set of infected edges, in the $H$-bootstrap process, at each time step, we infect an edge $e$ if it completes a new infected copy of $H$ in $G$.  This is obviously an instance of the hypergraph $\mathcal{H}$-bootstrap process for the $|E(H)|$-uniform hypergraph  $\mathcal{H}$ with $V(\mathcal{H}):=E(G)$, and hyperedges  that correspond to the copies of $H$ in $G$. }

\section*{Acknowledgement}
We wish to thank Jonathan Noel for bringing our attention on the reference~\cite{imbibition} \b{and the anonymous reviewer for pointing out the earlier application of polynomial method for bootstrap percolation in~\cite{Balister2018LineP}}. The second author wishes to thank Noga Alon for valuable comments and discussions.

\bibliographystyle{alpha}
\bibliography{percolation}

\end{document}